\documentclass[a4paper,12pt,reqno]{amsart}
\usepackage[utf8]{inputenc}
\usepackage{amsmath}
\usepackage{amsthm,amssymb,amsfonts}
\usepackage{graphicx}
\usepackage{hyperref}
\usepackage{url}

\usepackage{geometry}
\usepackage{array,bm}
\usepackage{color}
\usepackage{verbatim} 
\usepackage{xcolor}

\newtheorem{theorem}{Theorem}
\newtheorem{definition}[theorem]{Definition}

\newtheorem{remark}[theorem]{Remark}

\newtheorem{lemma}[theorem]{Lemma}

\makeatletter
\@namedef{subjclassname@2020}{\textup{2020} Mathematics Subject Classification}
\makeatother

\begin{document}
\title[Exponential convergence for ultrafast diffusion]{On the exponential convergence to equilibrium for ultrafast diffusion equations}

\author{Yi C. Huang} 
\address{School of Mathematical Sciences, Nanjing Normal University, Nanjing 210023, People's Republic of China}
\email{Yi.Huang.Analysis@gmail.com}
\urladdr{https://orcid.org/0000-0002-1297-7674}

\author{Xinhang Tong}
\address{School of Mathematical Sciences, Nanjing Normal University, Nanjing 210023, People's Republic of China}
\email{letterwoodtxh@gmail.com}

\date{\today} 


\thanks{Research of YCH is partially supported by the National NSF grant of China (no. 11801274), 
the JSPS Invitational Fellowship for Research in Japan (no. S24040), and the Open Projects from Yunnan Normal University (no. YNNUMA2403) and Soochow University (no. SDGC2418). 
Both authors thank Professor Max Fathi for sharing generously his insights on this problem.
YCH thanks in particular Jian-Yang Zhang (Shanghai) and Yuzhe Zhu (Chicago) for helpful communications,
and Qi Guo (Beijing) for hospitality at Renmin University where this manuscript is finalised.}

\maketitle
\begin{abstract}
We propose a simple proof of the exponential convergence to equilibrium for ultrafast diffusion equations in $\mathbb{R}^n$. 
Our approach, based on the direct use of Poincar\'e inequality, gets rid of the optimal transport arguments used in \cite{fathi2025}
which are valid for Gaussian-excluded one-dimensional weights. 
This simplification allows us to extend their results to Gaussian measures in higher dimensions.
\end{abstract}

\section{Introduction}

In \cite{Iacobelli19}, M. Iacobelli obtained the exponential convergence for the following so-called ``ultrafast diffusion equation" 
(see for example \cite{Bonforte17}, \cite{Bonforte2017} and \cite{Iacobelli2019} for the relevant background materials) with periodic boundary conditions:
\begin{equation*}
\begin{cases}
\partial_t f(t,x) = -r \, \partial_x \left( f(t,x) \, \partial_x \left( \frac{\rho(x)}{f^{r+1}(t,x)} \right) \right) &\quad \text{on } \quad(0,\infty) \times [0,1], \\[6pt]
f(t,0) = f(t,1) & \quad\text{on }\quad (0,\infty),
\end{cases}
\end{equation*}
where $r>1$ and both $\rho>0$ and $f(t, \cdot)$ are probability densities on $[0,1]$.
Recently in \cite{fathi2025}, M. Fathi and M. Iacobelli extended the results of \cite{Iacobelli19} to the non-compact case (namely, the whole real line $\mathbb{R}$). 
They considered the solutions $f(t, \cdot)$ of  
\begin{equation*}
\partial_t f(t,x) = -r \partial_x \left( f(t,x)  \partial_x \left( \frac{\rho(x)}{f^{r+1}(t,x)} \right) \right) \quad\text{on }\quad (0,\infty) \times \mathbb{R},
\end{equation*}
where $r>1$ and $\rho>0$ is a probability density on $\mathbb{R}$.
However, the approach in \cite{fathi2025} relies crucially on the optimal transport arguments which are valid for restrictive (Gaussian-excluded) one-dimensional weights. 
In this paper, with a rather simple argument, we extend the results of \cite{fathi2025} to Gaussian measures in higher dimensions. 

For the purpose as mentioned above, we shall consider the following (higher-dimensional) ultrafast diffusion equation with initial data $f_{0}(x):=f(0,x)$,
\begin{equation}\label{UFDE}
\partial_t f(t,x) = -r \, \mathrm{div}_{x}\left(f(t ,x) \nabla_{x} \left(\frac{\rho(x)}{f^{r+1}(t, x)}\right)\right) \quad \text{on } \quad(0, \infty) \times \mathbb{R}^n,
\end{equation}
where $r>1$ and $\rho>0$ is a probability density on $\mathbb{R}^n$.
According to the framework laid out in \cite{Iacobelli2019} and \cite{fathi2025}, one is led to study the following functional difference
$$
\mathcal{F} _{\rho}[f]-\mathcal{F} _{\rho}[m]:=\int_{\mathbb{R}^n}\frac{\rho(x)}{f^r(x)}dx-\int_{\mathbb{R}^n}\frac{\rho(x)}{m^r(x)}dx,
$$ 
where, for the solution $f$ we abuse the notation by dropping out the $t$ variable, and
$$m:=\gamma\rho^{1/(r+1)}\quad\text{ with}\quad \gamma :=\frac{1}{\int_{\mathbb{R}^n}\rho^{1/(r+1)}dx}.$$ 
Here $m$ is the equilibrium density of \eqref{UFDE} and $\gamma$ is the normalisation constant.

We have the following natural two-sided bounds. 

\begin{lemma}\label{asy}
Let $\mathcal{P}(\mathbb{R}^n)$ denote the space of probability densities on $\mathbb{R}^n$.
Assume that 
$$
f_0\in\mathcal{P}_{c,C}:=\{g\in\mathcal{P}(\mathbb{R}^n):cm\le g\le Cm\}, \quad C\geq1\geq c>0.
$$ 
Then there exist constants $k_1$ and $k_2$, depending on $r$, $c$, $C$ and $\gamma$, such that
\begin{equation*}\label{twoside}
\begin{aligned}
&k_1\int_{\mathbb{R}^n} \left|\frac{f}{m}-1\right|^2 mdx \le\mathcal{F} _{\rho}[f]-\mathcal{F} _{\rho}[m]\le k_2\int_{\mathbb{R}^n} \left|\frac{f}{m}-1\right|^2 mdx.
\end{aligned}
\end{equation*}
\end{lemma}

\begin{remark} \label{rem}
Under the assumption $f_0\in\mathcal{P}_{c,C}$, 
the existence and uniqueness of a solution $f(t)$ to \eqref{UFDE} with $f(t,\cdot)\in\mathcal{P}_{c,C}$ for all $t>0$ will be justified in Appendix \ref{exuni}. 
\end{remark}

As already well observed in \cite{Iacobelli2019} and \cite{fathi2025}, 
the exponential convergence of solutions $f(t, \cdot)$ to equilibrium $m$ can be reduced to the following crucial estimate
\begin{equation}\label{control}
\mathcal{F} _{\rho}[f]-\mathcal{F} _{\rho}[m] \le K\,I_{\rho}[f], \quad\forall f \in\mathcal{P}_{c,C},
\end{equation}
where $K$ is a constant and, by letting $u=f/m$,
$$I_{\rho}[f]:=-\frac{d}{dt}\mathcal{F} _{\rho}[f]=r^2\int_{\mathbb{R}^n} u|\nabla (u^{-(r+1)})|^2 mdx.$$ 
Given above lemma and in order to establish \eqref{control}, it is natural to invoke

\begin{definition}[Poincar\'e Inequality]
Let $m$ be a  probability measure on $\mathbb{R}^n$.
We say that the Poincar\'e inequality (associated to $m$) holds, with constant $C_P=C_P(m)$, if
\begin{equation*}
\int_{\mathbb{R}^n} g^2m dx-\left(\int_{\mathbb{R}^n} gm dx\right)^2\le C_P \int_{\mathbb{R}^n} |\nabla g|^2 mdx,\quad \forall \,g.
\end{equation*}
\end{definition}

\begin{remark}
Log-concave measures (in particular, Gaussian measures) satisfy the Poincar\'e inequality (see Bobkov \cite{Bobkov99}).
Certain non-log-concave measures also satisfy the Poincar\'e inequality (see \cite[Theorem 3]{barthe13} and \cite[Theorem 3.2]{bonnefont2021}).
\end{remark}

Now we are able to state the main contribution of this paper.  

\begin{theorem}\label{main}
The estimate \eqref{control} holds with
$$
K=\frac{C_P C^{2r+3}}{2r(r+1)\gamma^{r+1}c^{r+2}}.
$$ 
\end{theorem}

Then, by Gronwall type arguments in conjunction with Lemma \ref{asy}, we obtain the exponential convergence to equilibrium for ultrafast diffusion equations in $\mathbb{R}^n$.

\begin{theorem}\label{convergence}
Let $f(t)$ be the solution to \eqref{UFDE} with $f_{0}\in\mathcal{P}_{c,C}$. 
Then for all $t>0$,
\begin{equation*}
\int_{\mathbb{R}^n} \left| \frac{f(t)}{m}-1\right|^2 mdx\le \frac{1}{k_1}\left(\int_{\mathbb{R}^n}\frac{\rho(x)}{f_0^r(x)} dx-\frac{1}{\gamma^{r+1}}\right)\cdot e^{-\frac{1}{K}t},
\end{equation*}
where $k_1$ is given in Lemma \ref{asy} and $K$ is given in Theorem \ref{main}.
 \end{theorem}

Already for $m$ being the Gaussians, our result answers the open problems in \cite{fathi2025}.
Also, for compactly supported $f_0$ and $m$, Theorem \ref{convergence} covers \cite[Theorem 1.4]{Iacobelli2019}.

\section{Proofs}
\begin{proof}[Proof of Lemma \ref{asy}]
Recall $u=\frac{f}{m}$. By a simple calculation, we obtain
\begin{equation}\label{simple}
\begin{aligned}
\mathcal{F} _{\rho}[f]-\mathcal{F} _{\rho}[m]
&=\frac{1}{\gamma^{r+1}}\int_{\mathbb{R}^n} \left(u^{-r}-1\right)m dx\\
&=\frac{1}{\gamma^{r+1}}\int_{\mathbb{R}^n} \left(u^{-r}-1+r\left(u-1\right)\right)m dx.
\end{aligned}
\end{equation}
Considering the Taylor expansion of $x^{-r}$ at $x=1$, \eqref{simple} can be rewritten as
$$
\mathcal{F} _{\rho}[f]-\mathcal{F} _{\rho}[m]=\frac{r(r+1)}{2\gamma^{r+1}}\int_{\mathbb{R}^n} \theta^{-(r+2)}|u-1|^2 m dx,
$$
where $\theta\in(\min\{1,u\}, \max\{1,u\})$. Since $0<c\le u\le C$ (see Remark \ref{rem}), we obtain
\begin{equation*}\label{twobound}
\begin{aligned}
&\frac{r(r+1)}{2\gamma^{r+1}}C^{-(r+2)}\int_{\mathbb{R}^n} \left|u-1\right|^2 mdx\\
&\qquad\qquad\le\mathcal{F} _{\rho}[f]-\mathcal{F} _{\rho}[m]\\
&\qquad\qquad\le\frac{r(r+1)}{2\gamma^{r+1}}c^{-(r+2)}\int_{\mathbb{R}^n} \left|u-1\right|^2 mdx.
\end{aligned}
\end{equation*}
Here we also used implicitly $C\geq1\geq c$. This proves the lemma.
\end{proof}

\begin{proof}[Proof of Theorem \ref{main}]
Applying Poincar\'e inequality to $g=u-1=\frac{f}{m}-1$, and noting
$$\int_{\mathbb{R}^n} gm =\int_{\mathbb{R}^n} f-\int_{\mathbb{R}^n} m=1-1=0,$$ 
we obtain (upon tracing the expression of $k_2$ in the proof of Lemma \ref{asy})
$$
\mathcal{F} _{\rho}[f]-\mathcal{F} _{\rho}[m]\le C_P\frac{r(r+1)}{2\gamma^{r+1}}c^{-(r+2)}\int_{\mathbb{R}^n} \left|\nabla u\right|^2 mdx.
$$
Since $0<c\le u\le C$ (see Remark \ref{rem}), we also have
$$
\int_{\mathbb{R}^n} \left|\nabla u\right|^2 mdx\le \frac{C^{2r+3}}{(r+1)^2}\int_{\mathbb{R}^n} u|\nabla (u^{-(r+1)})|^2 mdx=\frac{C^{2r+3}}{r^2(r+1)^2} I_{\rho}[f].
$$
Collecting the above two inequalities then proves the theorem.
\end{proof}


\appendix

\section{On the existence and uniqueness of solutions}\label{exuni}

Here adapting \cite[Appendix A]{fathi2025} for the one-dimensional case, 
we show the well-posedness of the ultrafast diffusion equation \eqref{UFDE} in $\mathbb R^n$ in the class $\mathcal{P}_{c,C}$.

Let $f_{0}\in\mathcal{P}_{c,C}$ for some $c,C>0$ and write $m=e^{-V}$.
Given $k\ge 1$, we consider
$$
V_{k}:=
\begin{cases}
a_k V, & \text{in }B(0,k),\\[6pt]
+\infty,  & \text{otherwise},
\end{cases}
$$
where $a_k>0$ is chosen such that $m_k:=e^{-V_k}$ is a probability density in $\mathbb{R}^n$, 
or more precisely, in $B(0,k):=\{x\in\mathbb{R}^n:|x|\le k\}$. Also, we consider the initial data
$$
f_{0}^{k}:=
\begin{cases}
b_k f_0, & \text{in }B(0,k),\\[6pt]
0, & \text{otherwise},
\end{cases}
$$
where $b_k>0$ is chosen such that $f_{0}^{k}\in\mathcal{P}(B(0,k))$ (that is, $f_{0}^{k}$ is a probability density in $B(0,k)$). 
Note that $a_k$, $b_k\to 1$ as $k\to +\infty$, hence
$$
f_{0}^{k}\in\mathcal{P}_{c_k,C_k}^{k}:=\{g\in\mathcal{P}(B(0,k)):c_k m_k\le g\le C_k m_k\},
$$
where $c_k\to c$ and $C_k\to C$ as $k\to +\infty$.

Now, since $m_k$ is compactly supported, applying \cite[Theorem 1.2]{Iacobelli2019} 
we obtain the existence and uniqueness of the solution $f^{k}(t)\in\mathcal{P}^{k}_{c_k, C_k}$ to the ultrafast diffusion equation posed on $B(0, k)$ 
with Neumann boundary conditions and initial data $f_{0}^k$. 
In the process we have to invoke the maximum principle (see \cite[Corollary 3.8]{Iacobelli2019}) to guarantee that 
$f^{k}(t)$ inherits faithfully the constants $c_k$ and $C_k$ from $f_{0}^k\in\mathcal{P}^{k}_{c_k, C_k}$.
Moreover, for any $R<k$, we obtain from \cite[Lemma 3.3]{Iacobelli2019} the compactness of the family $\{f^k|_{B(0, R)}\}$. 
Hence, recalling that $c_k\to c$ and $C_k\to C$ as $k\to +\infty$,
by a diagonal argument we see that $f^{k}$ converges (up to a subsequence) to $f\in\mathcal{P}_{c,C}$, where $f(0)=f_0\in\mathcal{P}_{c,C}$ and $f(t)$ solves \eqref{UFDE}.
Also, upon localisation of $f_0$, the uniqueness of $f$ follows from the $L^1$-contractivity (see \cite[Proposition 3.6]{Iacobelli2019}).

\bigskip

\section*{\textbf{Compliance with ethical standards}}

\bigskip

\textbf{Conflict of interest} The authors have no known competing financial interests or personal relationships that could have appeared to influence this reported work.

\bigskip

\textbf{Availability of data and material} Not applicable.

\bigskip

\bibliographystyle{alpha}
\bibliography{newapproach}

\end{document}